\documentclass[12pt,reqno]{amsart}
\usepackage{mathrsfs}

\font\bbbld=msbm10 scaled\magstephalf

\usepackage{hyperref}

\usepackage{graphicx}

\usepackage{xcolor}

\newcommand{\bpartial}{\bar{\partial}}

\def \a{\alpha}

\def \b{\beta}
\def \p{\partial}

\newcommand{\bfR}{\hbox{\bbbld R}}

\newtheorem{theorem}{Theorem}[section]
\newtheorem{lemma}[theorem]{Lemma}

 \theoremstyle{definition}

\theoremstyle{remark}

\numberwithin{equation}{section}

\begin{document}
\setlength{\baselineskip}{1.2\baselineskip}

\title[On uniform estimate of complex elliptic equations]
{On uniform estimate of complex elliptic equations on closed Hermitian manifolds}

\author{Wei Sun}

\address{Shanghai Center for Mathematical Sciences, Fudan University}
\email{sunwei\_math@fudan.edu.cn}

\begin{abstract}
In this paper, we study Hessian equations and complex quotient equations on closed Hermitian manifolds. We directly derive the uniform estimate for the admissible solution. As an application, we solve general Hessian equations on closed K\"ahler manifolds.
\end{abstract}

\maketitle

\section{Introduction}
\label{gche-int}
\setcounter{equation}{0}
\medskip

Let $(M^n, \omega)$ be a compact Hermitian manifold of complex dimension $n\geq 2$ and $\chi$ a smooth real $(1,1)$ form on $M$. Write $\omega$ and $\chi$ respectively as
\[
\omega = {\sqrt{-1}} \sum_{i,j} g_{i\bar j} dz^i \wedge d\bar z^j\quad \text{ and }\quad
\chi = \sqrt{-1} \sum_{i,j} \chi_{i\bar j} dz^i \wedge d\bar z^j .
\]
We denote $\chi_u : = \chi + {\sqrt{-1}} \p\bpartial u$.


In the seminal paper~\cite{Yau78}, Yau proved the Calabi conjecture by solving the complex Monge-Amp\`ere equation on closed K\"ahler manifolds,
\begin{equation}
	\omega^n_u = f \omega^n, \qquad f > 0 \text{ on } M. 
\end{equation}
Since then, remarkable progress and geometric applications have been made. The Hermitian case was later solved by Tosatti and  Weinkove~\cite{TWv10b}. In the setting of moment maps, Donaldson~\cite{Donaldson99a} proposed an equation on closed K\"ahler manifolds
\begin{equation}
	\chi^n_u = f \chi^{n - 1}_u \wedge \omega, \qquad f > 0 \text{ on } M. 
\end{equation}
Surprisingly, in the study of Mabuchi enery, Chen~\cite{Chen00b} found the same equation in the critical case. In addition,  Hou,  Ma and  Wu~\cite{HouMaWu10} considered the Hessian equation on closed K\"ahler manifolds, for $2 \leq \a \leq  n - 1$
\begin{equation}
	\omega^\a_u \wedge \omega^{n - \a} = f \omega^n, \qquad f > 0 \text{ on } M,
\end{equation}
and proved a second order estimate. Following this estimate, Dinew and Kolodziej~\cite{DK} applied a blow-up argument to obtain the gradient estimate and consequently solved the Hessian equation. These equations attract much attention in complex geometry and analysis, and are still fast-developing subjects nowadays.


Indeed, all these equations can be included in the same class. For a smooth positive real function $f$ on $M$, we are concerned with the following complex equations,
\begin{enumerate}
\item {\em The complex $\a$-Hessian equation.} For $2 \leq \a \leq n$ and $\chi \in \Gamma^\a_\omega$,
\begin{equation}
\label{hessian-equation}
	\chi^\a_u \wedge \omega^{n - \a} = f \omega^n ,\qquad \text{with } \chi_u \in \Gamma^\a_\omega.
\end{equation}

\item {\em The complex $(\a,\b)$-quotient equation.} For $ 1 \leq \b < \a \leq n$ and $\chi \in \Gamma^\a_\omega$,
\begin{equation}
\label{quotient-equation}
	\chi^\a_u \wedge \omega^{n - \a} = f \chi^\b_u \wedge \omega^{n - \b} ,\qquad \text{with } \chi_u \in \Gamma^\a_\omega.
\end{equation}
Following \cite{SW08}, \cite{FLM11} and \cite{GSun12}, we define the cone class
\begin{equation}
\label{cone-condition}
	\mathscr{C}_{\a,\b} (f) := \{[\chi]\, : \exists \chi' \in \Gamma^\a_\omega \cap [\chi],\, \a \chi'^{\a - 1} \wedge \omega^{n - \a} > \b f \chi'^{\b - 1} \wedge \omega^{n - \b}\} .
\end{equation}
If $[\chi] \in \mathscr{C}_{\a,\b} (f)$, we say that $\chi$ satisfies the cone condition for equation~\eqref{quotient-equation} with respect to $f$. We say that $\chi'$ is in the cone subject to $f$.

\end{enumerate}
Here $\Gamma^\a_\omega$ is the set of all the real $(1, 1)$ forms whose eigenvalue set with respect to $\omega$ belong to $\a$-positive cone $\Gamma^\a \subset \bfR^n$. In fact, complex $\a$-Hessian equation can be treated as a particular case of complex quotient equations when $\beta = 0$.  However, complex Hessian equation has a natural strong cone condition, i.e. $\mathscr{C}_{\a,0} (f) = \{[\chi]\, :\, \Gamma^\a_\omega \cap [\chi] \neq \emptyset\}$, which is independent on $f$. This makes the approaches for Hessian equations easier than those for quotient equations.

In the work of the author~\cite{Sun2014e}, we provide a direct uniform estimate for the two types of equations on closed K\"ahler manifolds. In this paper, we further develop the technique in \cite{Sun2014e} and apply a key lemma by Zhang~\cite{Zhang2015} to extend the estimate to general Hermitian cases. The lemma in \cite{Zhang2015} helps us to improve our result in an early version of this paper. Our main result is as follows.
\begin{theorem}
\label{main-theorem}
Let $(M,g)$ be a closed Hermitian manifold of complex dimension $n \geq 2$ and $\chi$ a smooth real $(1,1)$ form. Let $u$ be a smooth solution to either \eqref{hessian-equation} or \eqref{quotient-equation} with the corresponding cone condition. 
Then there is a uniform estimate for $u$ depending only on $(M, \omega)$, $\chi$ and $f$.
\end{theorem}

It is worth a mention that the estimate of the quotient equation depends on the gradient of $f$ at most, while that in \cite{Sun2014e} needs the data of derivatives up to second order. Zhang also independently solved the complex Hessian equations when $\chi = \omega$ with a direct uniform estimate. Very recently, Zhang have informed the author that using G\r{a}ding's inequality, his method and result can be extended to more general cases. Moreover, Sz\'ekelyhidi~\cite{Szekelyhidi2014b} was able to solve the complex Hessian equations with a uniform estimate based on the approach of Blocki~\cite{Blocki2005a}.


Generally, the uniform estimate is the most difficult one on closed manifolds. When we are able to prove it, it is natural to consider the solvability.
As an application of the uniform estimate, we study complex Hessian equations on closed K\"ahler manifolds.
\begin{theorem}
\label{theorem-hessian}
Let $(M,g)$ be a closed K\"ahler manifold of complex dimension $n \geq 2$ and $\chi$ a smooth real $(1,1)$ form in $\Gamma^\a_\omega$. 
Then there exists a unique smooth real function $u$ and a unique real number $b$ such that
\begin{equation}
	\chi^\a_u \wedge \omega^{n - \a} = e^b f \omega^n ,\qquad \text{with } \chi_u \in \Gamma^\a_\omega \text{ and } \sup_M u = 0.
\end{equation}
\end{theorem}

When the equation is Monge-Amp\`ere type, Guan and the author~\cite{GSun12} treated the Dirichlet problem and derived the gradient estimate directly. For general cases, a blow-up argument is still necessary so far.

\bigskip
\noindent
{\bf Acknowledgements}\quad
The author would like to thank S{\l}awomir Dinew, Jixiang Fu and Bo Guan for comments and helpful suggestions.
The author also expresses his gratitude to Mulin Li  for helpful discussions.

\bigskip

\section{The uniform estimate}
\label{gche-proof}
\setcounter{equation}{0}
\medskip

It is sufficient to prove the inequality
\begin{equation}
\label{gche_auto_inequality}
	\Delta u > - C
\end{equation}
and
\begin{equation}
\label{gche_main_inequality}
	\int_M |\p e^{- \frac{p}{2} u}|^2_g \omega^n \leq C p\int_M e^{- p u} \omega^n
\end{equation}
for $p$ large enough. Without loss of generality, we may assume $p \gg 1$ throughout this section. The first inequality~\eqref{gche_auto_inequality} is naturally satisfied, and thus we only need to verify the second inequality~\eqref{gche_main_inequality}. We refer the readers to \cite{TWv10a},~\cite{TWv10b} and~\cite{Yau78} for more details.

For general Hermitian manifolds, we need the following lemma by Zhang~\cite{Zhang2015}.
\begin{lemma}
\label{gche-lemma-polynomial}
Suppose that $\lambda \in \Gamma^\a$, $3 \leq \a \leq n$ and $\lambda_1 \geq \lambda_2 \geq \cdots \geq \lambda_n$, then there exists a uniform  positive constant $C$ such that for $0 \leq i \leq \a - 2$ and $1 \leq j \leq n$,
\begin{equation}
	|\lambda_{j_1} \lambda_{j_2} \cdots \lambda_{j_i}| \leq C S_{i;j} ( \boldsymbol{\lambda}),
\end{equation} 
where $1 \leq j_1 < j_2 < \cdots < j_i \leq n$ and $j_k \neq j$ for all $1 \leq k \leq \a$ .
\end{lemma}

The following lemma was essentially proven by Sun~\cite{Sun2014e}.
\begin{lemma}
\label{gche-lemma-iteration}
Suppose that there are constants $\Lambda \geq \lambda > 0$ such that $\chi - \lambda \omega, \Lambda \omega - \chi \in \Gamma^\a_\omega$. For $2 \leq k \leq \a$, we have
\begin{equation}
\begin{aligned}
	&  \int^1_0 \left(\int_M e^{- p u} \sqrt{- 1} \p u \wedge \bpartial u \wedge \chi^{k - 1}_{t u} \wedge \omega^{n - k} \right) dt \\
	&\qquad\geq \frac{(k - 1)\lambda}{k} \int^1_0 \left(\int_M e^{- p u} \sqrt{- 1} \p u \wedge \bpartial u \wedge \chi^{k - 2}_{t u} \wedge \omega^{n - k + 1} \right) dt . 
\end{aligned}
\end{equation}
and
\begin{equation}
	 \int^1_0 \left(\int_M e^{- p u} \chi^k_{t u} \wedge \omega^{n - k}\right) dt \geq \frac{k\lambda}{k + 1} \int^1_0 \left(\int_M e^{- p u} \chi^{k - 1}_{t u} \wedge \omega^{n - k + 1}\right) dt  .
\end{equation}
\end{lemma}
\begin{proof}

 Using intergration by parts and G\r{a}rding's inequality, for $2 \leq k \leq \a$
\begin{equation}
\begin{aligned}
	& \int^1_0 \left(\int_M e^{- p u} \sqrt{- 1} \p u \wedge \bpartial u \wedge \chi^{k - 1}_{t u} \wedge \omega^{n - k} \right) dt \\
	\geq& \lambda \int^1_0 \left(\int_M e^{- p u} \sqrt{- 1} \p u \wedge \bpartial u \wedge \chi^{k - 2}_{t u} \wedge \omega^{n - k + 1} \right) dt \\
	& + \frac{1}{k - 1} \int^1_0 \left(\int_M e^{- p u} \sqrt{ - 1} \p u \wedge \bpartial u \wedge t \frac{d}{dt} \chi^{k - 1}_{t u} \wedge \omega^{n - k}\right) dt \\
	=& \lambda \int^1_0 \left(\int_M e^{- p u} \sqrt{- 1} \p u \wedge \bpartial u \wedge \chi^{k - 2}_{t u} \wedge \omega^{n - k + 1} \right) dt \\
	& + \frac{1}{k - 1} \int_M e^{- p u} \sqrt{- 1} \p u \wedge \bpartial u \wedge \chi^{k - 1}_u \wedge \omega^{n - k} \\
	& - \frac{1}{k - 1} \int^1_0 \left(\int_M e^{- p u} \sqrt{- 1} \p u \wedge \bpartial u \wedge \chi^{k - 1}_{t u} \wedge \omega^{n - k}\right) dt .
\end{aligned}
\end{equation}
Also, we have
\begin{equation}
\begin{aligned}
	&\, \int^1_0 \left(\int_M e^{- p u} \chi^k_{t u} \wedge \omega^{n - k}\right) dt \\
	\geq&\, \lambda \int^1_0 \left(\int_M e^{- p u} \chi^{k - 1}_{t u} \wedge \omega^{n - k + 1}\right) dt + \frac{1}{k} \int^1_0 \left(\int_M e^{- p u} t \frac{d}{d t} \chi^k_{t u} \wedge \omega^{n - k}\right) dt \\
	=&\,  \lambda \int^1_0 \left(\int_M e^{- p u} \chi^{k - 1}_{t u} \wedge \omega^{n - k + 1}\right) dt + \frac{1}{k} \int_M e^{- p u} \chi^k_u \wedge \omega^{n - k} \\
	&\, - \frac{1}{k} \int^1_0 \left(\int_M e^{- p u} \chi^k_{t u} \wedge \omega^{n - k} \right)dt .
\end{aligned}
\end{equation}

\end{proof}

We shall use Lemma~\ref{gche-lemma-polynomial} and Lemma~\ref{gche-lemma-iteration} to deal with some troublesome terms due to torsion. 
\begin{lemma}
\label{gche-lemma-1}
There is a uniform constant $C > 0$ such that
\begin{equation}
\label{gche-lemma-1-inequality}
\begin{aligned}
	&\, \left|\int^1_0 \left(\int_M e^{- p u} \chi^{\a - 1}_{t u} \wedge \sqrt{- 1} \p \bpartial \omega^{n - \a} \right) dt \right| \\
	\leq&\, C p \int^1_0 \left(\int_M e^{- p u} \sqrt{- 1} \p u \wedge \bpartial u \wedge \chi^{\a - 2}_{t u} \wedge \omega^{n - \a + 1}\right) dt \\
	&\,  + C \int^1_0 \left(\int_M e^{- p u} \chi^{\a - 2}_{t u} \wedge \omega^{n - \a + 2}\right) dt  .
\end{aligned}
\end{equation}
\end{lemma}
\begin{proof}
Rewriting,
\begin{equation}
\label{gche-lemma-1-equality-1}
\begin{aligned}
	&\, \int^1_0 \left(\int_M e^{- p u} \chi^{\a - 1}_{t u} \wedge \sqrt{- 1} \p \bpartial \omega^{n - \a} \right) dt \\
	=&\, \int^1_0 \left(\int_M e^{- p u} \chi^{\a - 2}_{t u} \wedge \chi \wedge \sqrt{- 1} \p \bpartial \omega^{n - \a} \right) dt \\
	&\, - \int^1_0 t \left(\int_M e^{- p u}  \p\bpartial u \wedge \chi^{\a - 2}_{t u} \wedge  \p \bpartial \omega^{n - \a} \right) dt.
\end{aligned}
\end{equation}
We apply integration by parts to the second term,
\begin{equation}
\label{gche-lemma-1-equality-2}
\begin{aligned}
	&\,   \int^1_0 t \left(\int_M e^{- p u}  \p\bpartial u \wedge \chi^{\a - 2}_{t u} \wedge \p \bpartial \omega^{n - \a} \right) dt \\
	=&\,  p \int^1_0 t \left(\int_M e^{- p u} \p u \wedge \bpartial u \wedge \chi^{\a - 2}_{t u} \wedge  \p\bpartial \omega^{n - \a}\right) dt \\
	&\, + (\a - 2) \int^1_0 t \left(\int_M e^{- p u}  \bpartial u \wedge \p \chi \wedge \chi^{\a - 3}_{t u} \wedge  \p\bpartial \omega^{n - \a}\right) dt .
\end{aligned}
\end{equation}
For a fixed $t$, we express $X = \chi_{t u}$. For any point $p\in M$, we can find a coordinate chart around $p$ such that $g_{i\bar j} = \delta_{ij}$ and $X_{i\bar j}$ is diagonal at $p$.  
In this paper, we call such local coordinates normal coordinate charts around $p$.
Then we pointwise control the first term in the right of \eqref{gche-lemma-1-equality-1}
\begin{equation}
\begin{aligned}
	\left| \frac{\chi^{\a - 2}_{t u} \wedge \chi \wedge \sqrt{- 1} \p \bpartial \omega^{n - \a}}{\omega^n} \right| 
	\leq&\, \sum_{i_1 < \cdots <i_{\a - 2}} \left( |R_{i_1 \cdots i_{\a - 2}}| \prod_{j \in \{i_1, \cdots, i_{\a - 2}\}} |X_{j\bar j}| \right) \\
	\leq&\, \frac{C \chi^{\a - 2}_{t u} \wedge \omega^{n - \a + 2}}{\omega^n} .
\end{aligned}
\end{equation}
We pointwise control the first term in \eqref{gche-lemma-1-equality-2},
\begin{equation}
\label{gche-lemma-1-equality-3}
\begin{aligned}
	&\, \left|\frac{ \p u \wedge \bpartial u \wedge \chi^{\a - 2}_{t u} \wedge  \p\bpartial \omega^{n - \a}}{\omega^n}\right| \\
	\leq&\, \sum_{\substack{i_1 < \cdots <i_{\a - 2} \\ j,k \neq i_1 ,  \cdots , i_{\a - 2}}} \Bigg(\left|R_{\bar j k i_1 \cdots i_{\a - 2}}\right| \left|u_j\right| \left| u_{\bar k}\right| \prod_{l \in \{i_1,\cdots,i_{\a - 2}\}} \left|X_{l \bar l}\right| \Bigg) .
\end{aligned}
\end{equation}
By Schwarz inequality and Lemma~\ref{gche-lemma-polynomial}, we have
\begin{equation}
\label{gche-lemma-1-equality-4}
\begin{aligned}
	\left|\frac{ \p u \wedge \bpartial u \wedge \chi^{\a - 2}_{t u} \wedge  \p\bpartial \omega^{n - \a}}{\omega^n}\right| 
	\leq&\, C \sum_{\substack{i_1  < \cdots <i_{\a - 2} \\ j \neq i_1 ,  \cdots , i_{\a - 2}}} \Bigg( |u_j|^2  \prod_{l \in \{i_1,\cdots,i_{\a - 2}\}} \left|X_{l \bar l}\right| \Bigg) \\
	\leq&\,  \frac{C \sqrt{-1} \p u \wedge \bpartial u \wedge \chi^{\a - 2}_{t u} \wedge \omega^{n - \a + 1}}{\omega^n} .
\end{aligned}
\end{equation}
We pointwise control the second term in \eqref{gche-lemma-1-equality-2},
\begin{equation}
\label{gche-lemma-1-equality-5}
\begin{aligned}
	&\, \left| \frac{ \bpartial u \wedge \p \chi \wedge \chi^{\a - 3}_{t u} \wedge  \p\bpartial \omega^{n - \a}}{\omega^n} \right| \\
	\leq&\, \sum_{\substack{i_1  < \cdots <i_{\a - 3} \\ j \neq i_1 ,  \cdots , i_{\a - 3}}} \Bigg(\left|R_{ j i_1 \cdots i_{\a - 3}}\right|| u_{\bar j}| \prod_{k \in \{i_1,\cdots,i_{\a - 3}\}} \left|X_{k \bar k}\right| \Bigg).
\end{aligned}
\end{equation}
By Schwarz inequality and Lemma~\ref{gche-lemma-polynomial}, we have
\begin{equation}
\label{gche-lemma-1-equality-6}
\begin{aligned}
	&\, \left| \frac{ \bpartial u \wedge \p \chi \wedge \chi^{\a - 3}_{t u} \wedge \p\bpartial \omega^{n - \a}}{\omega^n} \right| \\
	\leq&\, C \sum_{\substack{i_1 < \cdots <i_{\a - 3} \\ j \neq i_1 ,  \cdots , i_{\a - 3}}} \Bigg(\left(|u_j|^2 + 1\right) \prod_{k \in \{i_1,\cdots,i_{\a - 3}\}} \left|X_{k \bar k}\right| \Bigg) \\
	\leq&\, \frac{C \sqrt{- 1} \p u \wedge \bpartial u \wedge \chi^{\a - 3}_{t u} \wedge \omega^{n - \a + 2}}{\omega^n} + \frac{C \chi^{\a - 3}_{t u} \wedge \omega^{n - \a + 3}}{\omega^n} .
\end{aligned}
\end{equation}
Applying Lemma~\ref{gche-lemma-iteration}, the proof is finished.

\end{proof}


\begin{lemma}
\label{hessian-lemma-estimate}
Let $u$ be a smooth admissible solution of Hessian equation~\eqref{hessian-equation}. Then there are uniform constants $C$, $p_0$ such that for all $p \geq p_0$, inequality~\eqref{gche_main_inequality} holds true. 
\end{lemma}

\begin{proof}

Without loss of generality, it is reasonable to assume that $p \gg 1$. We consider the following inequality,
\begin{equation}
\label{gche-hessian-integral}
I:= \int_M e^{- p u} \left(\chi^\a_u \wedge \omega^{n - \a} - \chi^\a \wedge \omega^{n - \a}\right).
\end{equation}
It is easy to see that $I \leq C_1\int_M e^{- p u} \omega^n$ for some uniform constant $C_1$.

We rewrite the integral,
\begin{equation}
\label{gche-hessian-integral-new-form}
	I = \a \int^1_0\left(\int_M e^{- p u} \sqrt{- 1} \p\bpartial u \wedge \chi^{\a - 1}_{t u} \wedge \omega^{n - \a}\right) dt ,
\end{equation}
and then calculate it directly,
\begin{equation}
\begin{aligned}
	I
	&=  \a p \int^1_0 \left(\int_M e^{- p u} \sqrt{-1} \p u \wedge\bpartial u \wedge \chi^{\a - 1}_{t u} \wedge \omega^{n - \a}\right) dt \\
	&\quad + \a (\a - 1) \int^1_0 \left(\int_M e^{- p u} \sqrt{- 1} \bpartial u \wedge \p \chi \wedge \chi^{\a - 2}_{t u} \wedge \omega^{n - \a}\right) dt \\
	&\quad + \frac{\a}{p}\int^1_0 \left(\int_M e^{- p u} \sqrt{- 1} \bpartial \chi^{\a - 1}_{t u} \wedge \p \omega^{n - \a}\right) dt \\
	&\quad - \frac{\a}{p}\int^1_0 \left(\int_M e^{- p u}  \chi^{\a - 1}_{t u} \wedge \sqrt{- 1} \p \bpartial \omega^{n - \a} \right) dt .
\end{aligned}
\end{equation}
By integration by parts for the third term,
\begin{equation}
\label{gche-hessian-integral-1}
\begin{aligned}
	I &=  \a p \int^1_0 \left(\int_M e^{- p u} \sqrt{-1} \p u \wedge\bpartial u \wedge \chi^{\a - 1}_{t u} \wedge \omega^{n - \a}\right) dt \\
	&\quad + \a (\a - 1) \int^1_0 \left(\int_M e^{- p u} \sqrt{- 1} \bpartial u \wedge \p \chi \wedge \chi^{\a - 2}_{t u} \wedge \omega^{n - \a}\right) dt \\
	&\quad - \a (\a - 1) \int^1_0 \left(\int_M e^{- p u} \sqrt{- 1} \p u \wedge \bpartial \chi \wedge \chi^{\a - 2}_{t u} \wedge \omega^{n - \a}\right) dt \\
	&\quad + \frac{\a (\a - 1)}{p}\int^1_0 \left(\int_M e^{- p u} \sqrt{- 1} \p\bpartial \chi \wedge \chi^{\a - 2}_{t u} \wedge \omega^{n - \a}\right) dt \\
	&\quad - \frac{\a (\a - 1) (\a - 2)}{p}\int^1_0 \left(\int_M e^{- p u} \sqrt{- 1} \bpartial \chi \wedge \p \chi \wedge \chi^{\a - 3}_{t u} \wedge \omega^{n - \a}\right) dt\\
	&\quad - \frac{\a}{p}\int^1_0 \left(\int_M e^{- p u}  \chi^{\a - 1}_{t u} \wedge \sqrt{- 1} \p \bpartial \omega^{n - \a} \right) dt .
\end{aligned}
\end{equation}
We can pointwise control four terms in the right hand side of \eqref{gche-hessian-integral-1}. By Schwarz's inequality, for any $\epsilon > 0$ there is a uniform constant $C$ such that
\begin{equation}
\label{gche-hessian-inequality-1}
\begin{aligned}
    &\hspace{-2em} \left|\frac{\sqrt{-1} \left(\bpartial u \wedge \p \chi - \p u \wedge \bpartial \chi \right)\wedge \chi^{\a - 2}_{t u} \wedge \omega^{n - \a}}{\omega^n}\right| \\
	&\leq \frac{C}{\epsilon} \frac{\sqrt{-1} \p u \wedge \bpartial u \wedge \omega^{\a - 2}_{t u} \wedge \omega^{n - \a + 1}}{\omega^n} + \epsilon \frac{\chi^{\a - 2}_{t u} \wedge \omega^{n - \a + 2}}{\omega^n} .
\end{aligned}
\end{equation}
Also\begin{equation}
\label{gche-hessian-inequality-2}
	\left|\frac{\sqrt{- 1} \p\bpartial \chi \wedge \chi^{\a - 2}_{t u} \wedge \omega^{n - \a}}{\omega^n} \right| \leq C \frac{\chi^{\a - 2}_{t u} \wedge \omega^{n - \a}}{\omega^n} ,
\end{equation}
and
\begin{equation}
\label{gche-hessian-inequality-3}
	\left|\frac{\sqrt{- 1} \bpartial\chi \wedge \p\chi \wedge \chi^{\a - 3}_{t u} \wedge \omega^{n - \a}}{\omega^n}\right| \leq C \frac{\chi^{\a - 3}_{t u} \wedge \omega^{n - \a + 3}}{\omega^n} .
\end{equation}
Combining \eqref{gche-hessian-integral-1}, \eqref{gche-hessian-inequality-1}, \eqref{gche-hessian-inequality-2} and \eqref{gche-hessian-inequality-3},
\begin{equation}
\label{gche-hessian-integral-2}
\begin{aligned}
	I &\geq  \a p \int^1_0 \left(\int_M e^{- p u} \sqrt{-1} \p u \wedge\bpartial u \wedge \chi^{\a - 1}_{t u} \wedge \omega^{n - \a}\right) dt \\
	&\quad - \epsilon_1 p \int^1_0 \left(\int_M e^{- p u} \sqrt{- 1} \p u \wedge \bpartial u \wedge \chi^{\a - 2}_{t u} \wedge \omega^{n - \a + 1}\right) dt \\
	&\quad - \frac{C_2}{p}\int^1_0 \left(\int_M e^{- p u} \chi^{\a - 2}_{t u} \wedge \omega^{n - \a + 2}\right) dt \\
	&\quad - \frac{C_3}{p}\int^1_0 \left(\int_M e^{- p u} \chi^{\a - 3}_{t u} \wedge \omega^{n - \a + 3}\right) dt\\
	&\quad - \frac{\a}{p}\int^1_0 \left(\int_M e^{- p u}  \chi^{\a - 1}_{t u} \wedge \sqrt{- 1} \p \bpartial \omega^{n - \a} \right) dt .
\end{aligned}
\end{equation}
where $\epsilon_1$ is to be determined later. By Lemma~\ref{gche-lemma-iteration}, we know that if $\epsilon_1 \leq \frac{\lambda}{2}$,
\begin{equation}
\label{gche-hessian-integral-3}
\begin{aligned}
	I &\geq  \frac{\lambda p}{2} \int^1_0 \left(\int_M e^{- p u} \sqrt{- 1} \p u \wedge \bpartial u \wedge \chi^{\a - 2}_{t u} \wedge \omega^{n - \a + 1}\right) dt \\
	&\quad - \frac{C_4}{p}\int^1_0 \left(\int_M e^{- p u} \chi^{\a - 2}_{t u} \wedge \omega^{n - \a + 2}\right) dt \\
	&\quad - \frac{\a}{p}\int^1_0 \left(\int_M e^{- p u}  \chi^{\a - 1}_{t u} \wedge \sqrt{- 1} \p \bpartial \omega^{n - \a} \right) dt .
\end{aligned}
\end{equation}

If $\a = 2$, we calculate the last term in \eqref{gche-hessian-integral-3},
\begin{equation}
\begin{aligned}
	&\, \int^1_0 \left(\int_M e^{- p u}  \chi_{t u} \wedge \sqrt{- 1} \p \bpartial \omega^{n - 2} \right) dt \\
	=&\,  \int^1_0  \left(\int_M e^{- p u}  \chi \wedge \sqrt{- 1} \p \bpartial \omega^{n - 2} \right) dt \\
	&\, + \frac{p}{2} \int_M e^{- p u} \sqrt{- 1} \p u \wedge \bpartial u \wedge  \sqrt{- 1} \p \bpartial \omega^{n - 2} 
\end{aligned}
\end{equation}
and thus
\begin{equation}
\label{gche-hessian-integral-4}
\begin{aligned}
	I &\geq  \frac{\lambda p}{2} \int_M e^{- p u} \sqrt{- 1} \p u \wedge \bpartial u \wedge \omega^{n - 1} - \frac{C_4}{p} \int_M e^{- p u} \omega^{n} \\
	&\quad - \frac{2}{p}\int^1_0 \left(\int_M e^{- p u}  \chi_{t u} \wedge \sqrt{- 1} \p \bpartial \omega^{n - 2} \right) dt \\
	&\geq \frac{\lambda p}{3} \int_M e^{- p u} \sqrt{- 1} \p u \wedge \bpartial u \wedge \omega^{n - 1} - \frac{C_5}{p} \int_M e^{- p u} \omega^n
\end{aligned}
\end{equation}
if $p$ is large enough.

If $\a \geq 3$, we control the last term in  \eqref{gche-hessian-integral-3}. By Lemma~\ref{gche-lemma-1},
\begin{equation}
\label{gche-hessian-integral-5}
\begin{aligned}
	I &\geq  \frac{\lambda p}{2} \int^1_0 \left(\int_M e^{- p u} \sqrt{- 1} \p u \wedge \bpartial u \wedge \chi^{\a - 2}_{t u} \wedge \omega^{n - \a + 1}\right) dt \\
	&\quad - C_6 \int^1_0 \left(\int_M e^{- p u} \sqrt{- 1} \p u \wedge \bpartial u \wedge \chi^{\a - 2}_{t u} \wedge \omega^{n - \a + 1}\right) dt \\
	&\quad  - \frac{C_6}{p} \int^1_0 \left(\int_M e^{- p u} \chi^{\a - 2}_{t u} \wedge \omega^{n - \a + 2}\right) dt  .
\end{aligned}
\end{equation}
To control the last term in \eqref{gche-hessian-integral-5}, we compute
\begin{equation}
\label{gche-hessian-integral-6}
\begin{aligned}
	 & \frac{1}{p}\int^1_0 \left(\int_M e^{- p u} \chi^{\a - 2}_{t u} \wedge \omega^{n - \a + 2}\right) dt \\
	 =& (\a - 2) \int^1_0 \int^t_0 \left(\int_M e^{- p u} \sqrt{- 1} \p u \wedge \bpartial u \wedge \chi^{\a - 3}_{su} \wedge \omega^{n - \a + 2}\right) ds dt \\
	 & + \frac{\a - 2}{p} \int^1_0 \int^t_0 \left(\int_M e^{- p u} \sqrt{- 1} \bpartial u \wedge \p (\chi^{\a - 3}_{s u} \wedge \omega^{n - \a + 2})\right) ds dt \\
	 & + \frac{1}{p} \int_M e^{- p u} \chi^{\a - 2} \wedge \omega^{n - \a + 2} \\
	 \leq& C_7 \int^1_0 \left(\int_M e^{- p u} \sqrt{- 1} \p u \wedge \bpartial u \wedge \chi^{\a - 3}_{su} \wedge \omega^{n - \a + 2}\right) ds \\
	 & + \frac{1}{p^2} \int^1_0 \left(\int_M e^{- p u} \chi^{\a - 3}_{s u} \wedge \omega^{n - \a + 3}\right) ds  + \frac{1}{p} \int_M e^{- p u} \chi^{\a - 2} \wedge \omega^{n - \a + 2}.
\end{aligned} 
\end{equation}
Consequently, for sufficiently large $p$,
\begin{equation}
\label{gche-hessian-integral-7}
\begin{aligned}
	&\qquad \frac{1}{p}\int^1_0 \left(\int_M e^{- p u} \chi^{\a - 2}_{t u} \wedge \omega^{n - \a + 2}\right) dt \\
	&\leq 2 C_7   \int^1_0 \left(\int_M e^{- p u} \sqrt{- 1} \p u \wedge \bpartial u \wedge \chi^{\a - 3}_{t u} \wedge \omega^{n - \a + 2}\right) dt  \\
	&\hspace{6em} + \frac{2}{p} \int_M e^{- p u} \chi^{\a - 2} \wedge \omega^{n - \a + 2}
\end{aligned}
\end{equation}
and hence from \eqref{gche-hessian-integral-5},
\begin{equation}
\label{gche-hessian-integral-8}
\begin{aligned}
	I \geq \frac{\lambda p}{3} \int^1_0 \left(\int_M e^{- p u} \sqrt{- 1} \p u \wedge \bpartial u \wedge \chi^{\a - 2}_{t u} \wedge \omega^{n - \a + 1}\right) dt -\frac{C_8}{p} \int_M e^{- p u} \omega^n.
\end{aligned}
\end{equation}
As in \cite{Sun2014e}, we can show that for some $c_0 > 0$,
\begin{equation}
\label{gche-hessian-integral-9}
	\frac{\lambda p}{3} \int^1_0 \left(\int_M e^{- p u} \sqrt{- 1} \p u \wedge \bpartial u \wedge \chi^{\a - 2}_{t u} \wedge \omega^{n - \a + 1}\right) dt \geq \frac{4 c_0}{n p} \int_M |\p e^{-\frac{p}{2} u}|^2_g \omega^n .
\end{equation}

\end{proof}


\begin{lemma}
\label{quotient-lemma-estimate}
Let $u$ be a smooth admissible solution to quotient equation~\eqref{quotient-equation}. Then there are uniform constants $C$, $p_0$ such that for all $p \geq p_0$, inequality~\eqref{gche_main_inequality} holds true.
\end{lemma}

\begin{proof}

Without loss of generality, we can assume
\begin{equation}
	\a \chi^{\a - 1} \wedge \omega^{n - \a} > \b f \chi^{\b - 1} \wedge \omega^{n - \b} .
\end{equation}
Also, by the monotony of $S_\a/S_\b$, we have
\begin{equation}
	\a \chi^{\a - 1}_u \wedge \omega^{n - \a} > \b f \chi^{\b - 1}_u \wedge \omega^{n - \b}.
\end{equation}

We consider the integral
\begin{equation}
	I := \int_M e^{- p u} \left((\chi^\a_u \wedge \omega^{n - \a} - \chi^\a \wedge \omega^{n - \a}) - f (\chi^\b_u \wedge \omega^{n - \b} - \chi^\b \wedge \omega^{n - \b})\right) .
\end{equation}

It is not hard to see
\begin{equation}
	I \leq C_1 \int_M e^{- p u} \omega^n ,
\end{equation}
for some uniform constant $C_1 > 0$.

Computing $I$ directly,
\begin{equation}
\begin{aligned}
	I =&\, p \int^1_0 \left(\int_M e^{- p u} \sqrt{- 1} \p u  \wedge \bpartial u \wedge (\a \chi^{\a - 1}_{t u} \wedge \omega^{n - \a} - \b f \chi^{\b - 1}_{t u} \wedge \omega^{n - \b}) \right) dt \\
	&\,  + \a (\a - 1)\int^1_0 \left(\int_M e^{- p u} \sqrt{- 1} \bpartial u \wedge \p \chi \wedge \chi^{\a - 2}_{t u} \wedge \omega^{n - \a} \right) dt \\
	&\,  - \a (\a - 1)\int^1_0 \left(\int_M e^{- p u} \sqrt{- 1} \p u \wedge \bpartial \chi \wedge \chi^{\a - 2}_{t u} \wedge \omega^{n - \a} \right) dt \\
	&\,  + \frac{\a}{p} \int^1_0 \left(\int_M  e^{- p u} \sqrt{- 1} \p\bpartial \chi^{\a - 1}_{t u} \wedge \omega^{n - \a}\right) dt \\
	&\,  - \frac{\a}{p} \int^1_0 \left(\int_M e^{- p u} \chi^{\a - 1}_{t u} \wedge \sqrt{- 1} \p \bpartial \omega^{n - \a}\right) dt \\
	&\, - \beta \int^1_0 \left(\int_M e^{- p u} \sqrt{- 1} \bpartial u \wedge \p f \wedge \chi^{\b - 1}_{t u} \wedge \omega^{n - \b}\right) dt \\
	&\, - \beta \int^1_0 \left(\int_M e^{- p u} f \sqrt{- 1} \bpartial u \wedge \p (\chi^{\beta - 1}_{t u} \wedge \omega^{n - \beta})\right) dt .
\end{aligned}
\end{equation}
Therefore, as in the proof of Lemma~\ref{hessian-lemma-estimate},
\begin{equation}
\label{gche-quotient-integral-1}
\begin{aligned}
	I 
	&\geq p \int^1_0 \left(\int_M e^{- p u} \sqrt{- 1} \p u  \wedge \bpartial u \wedge (\a \chi^{\a - 1}_{t u} \wedge \omega^{n - \a} - \b f \chi^{\b - 1}_{t u} \wedge \omega^{n - \b}) \right) dt \\
	&\quad - \epsilon_1 p \int^1_0 \left(\int_M e^{- p u} \sqrt{- 1} \p u \wedge \bpartial u \wedge \chi^{\a - 2}_{t u} \wedge \omega^{n - \a + 1} \right) dt \\
	&\quad - \frac{C_2}{p} \int^1_0 \left(\int_M e^{- p u} \chi^{\a - 2}_{t u} \wedge \omega^{n - \a + 2}\right) dt \\
	&\quad - \frac{\a}{p} \int^1_0 \left(\int_M e^{- p u} \chi^{\a - 1}_{t u} \wedge  \sqrt{- 1} \p \bpartial \omega^{n - \a}\right) dt .
\end{aligned}
\end{equation}

If $\a = 2$, then $\b = 1$.
\begin{equation}
\label{gche-quotient-integral-2}
\begin{aligned}
	I &\geq p \int^1_0 \left(\int_M e^{- p u} \sqrt{- 1} \p u  \wedge \bpartial u \wedge (2 \chi_{t u} \wedge \omega^{n - 2} - f \omega^{n - 1}) \right) dt \\
	&\quad - \epsilon_1 p \int_M e^{- p u} \sqrt{- 1} \p u \wedge \bpartial u \wedge \omega^{n - 1} - \frac{C_2}{p} \int_M e^{- p u} \omega^{n}\\
	&\quad - \frac{2}{p} \int^1_0 \left(\int_M e^{- p u} \chi_{t u} \wedge \sqrt{- 1} \p \bpartial \omega^{n - 2}\right) dt 
\end{aligned}
\end{equation}
where
\begin{equation}
\label{gche-quotient-integral-3}
\begin{aligned}
	& - \frac{2}{p} \int^1_0 \left(\int_M e^{- p u} \chi_{t u} \wedge \sqrt{- 1} \p \bpartial \omega^{n - 2}\right) dt  \\
	= & - \frac{2}{p} \int_M e^{- p u} \chi \wedge \sqrt{- 1} \p \bpartial \omega^{n - 2} + \int_M  e^{- p u} \p u \wedge \bpartial u \wedge \p\bpartial \omega^{n - 2} \\
	\geq& - \frac{C_3}{p} \int_M e^{- p u} \omega^n - C_4 \int_M e^{- p u} \sqrt{- 1} \p u \wedge \bpartial u \wedge \omega^{n - 1} .
\end{aligned}
\end{equation}
As proven in \cite{Sun2014e}, there exists $c_1 > 0$,
\begin{equation}
\label{gche-quotient-integral-4}
\begin{aligned}
	&\int^1_0 \left(\int_M e^{- p u} \sqrt{- 1} \p u  \wedge \bpartial u \wedge (2 \chi_{t u} \wedge \omega^{n - 2} - f \omega^{n - 1}) \right) dt \\
	&\hspace{12em} \geq c_1 \int_M e^{- p u} \sqrt{- 1} \p u \wedge \bpartial u \wedge \omega^{n - 1}.
\end{aligned}
\end{equation}
Choosing $\epsilon_1 \leq \frac{c_1}{3}$ and then setting $p$ so large that $c_1 p > 6 C_4$, we have
\begin{equation}
\label{gche-quotient-integral-5}
	I \geq \frac{c_1 p}{2} \int_M e^{- p u} \sqrt{- 1} \p u \wedge \bpartial u \wedge \omega^{n - 1} - \frac{C_2 + C_3}{p} \int_M e^{- p u} \omega^n.
\end{equation}

If $\a \geq 3$, we use Lemma~\ref{gche-lemma-1} to control the last term in \eqref{gche-quotient-integral-1}.
\begin{equation}
\label{gche-quotient-integral-11}
\begin{aligned}
	I &= \int^1_0 \left(\int_M e^{- p u} \sqrt{- 1} \p\bpartial u \wedge (\a \chi^{\a - 1}_{t u} \wedge \omega^{n - \a} - \b f \chi^{\b - 1}_{t u} \wedge \omega^{n - \b}) \right) dt \\
	&\geq p \int^1_0 \left(\int_M e^{- p u} \sqrt{- 1} \p u  \wedge \bpartial u \wedge (\a \chi^{\a - 1}_{t u} \wedge \omega^{n - \a} - \b f \chi^{\b - 1}_{t u} \wedge \omega^{n - \b}) \right) dt \\
	&\quad - \epsilon_1 p \int^1_0 \left(\int_M e^{- p u} \sqrt{- 1} \p u \wedge \bpartial u \wedge \chi^{\a - 2}_{t u} \wedge \omega^{n - \a + 1} \right) dt \\
	&\quad - \frac{(C_2 + C_5)}{p} \int^1_0 \left(\int_M e^{- p u} \chi^{\a - 2}_{t u} \wedge \omega^{n - \a + 2}\right) dt \\
	&\quad - C _5 \int^1_0 \left(\int_M e^{- p u} \sqrt{- 1} \p u \wedge \bpartial u \wedge \chi^{\a - 2}_{t u} \wedge \omega^{n - \a + 1}\right) dt .
	\end{aligned}
\end{equation}
By \eqref{gche-hessian-integral-7},
\begin{equation}
\label{gche-quotient-integral-6}
\begin{aligned}
	I 
	&\geq p \int^1_0 \left(\int_M e^{- p u} \sqrt{- 1} \p u  \wedge \bpartial u \wedge (\a \chi^{\a - 1}_{t u} \wedge \omega^{n - \a} - \b f \chi^{\b - 1}_{t u} \wedge \omega^{n - \b}) \right) dt \\
	&\quad - \epsilon_1 p \int^1_0 \left(\int_M e^{- p u} \sqrt{- 1} \p u \wedge \bpartial u \wedge \chi^{\a - 2}_{t u} \wedge \omega^{n - \a + 1} \right) dt \\
	&\quad - C_6 \int^1_0 \left(\int_M e^{- p u} \sqrt{- 1} \p u \wedge \bpartial u \wedge \chi^{\a - 2}_{t u} \wedge \omega^{n - \a + 1} \right) dt - \frac{C_7}{p} \int_M e^{- p u} \omega^n .
\end{aligned}
\end{equation}
Note that $C_6$ and $C_7$ depend on $\epsilon_1$. 
The concavity of hyperbolic polynomials implies that
\begin{equation}
\label{gche-quotient-integral-7}
\begin{aligned}
	&\quad \int^1_0 \left(\int_M e^{- p u} \sqrt{- 1} \p u \wedge \bpartial u \wedge \chi^{\a - 2}_{t u} \wedge \omega^{n - \a + 1}\right) dt \\
	&\leq 2^{\a - 1} \int^{\frac{1}{2}}_0 \left(\int_M e^{- p u} \sqrt{- 1} \p u \wedge \bpartial u \wedge \chi^{\a - 2}_{t u} \wedge \omega^{n - \a + 1}\right) dt .
\end{aligned}
\end{equation} 
Moreover, the concavity of quotient functions implies that, there are constants $c_1 > 0$ and $c_2 > 0$ such that
\begin{equation}
\label{gche-quotient-integral-8}
\begin{aligned}
	&\qquad \int^1_0 \left(\int_M e^{- p u} \sqrt{- 1} \p u  \wedge \bpartial u \wedge (\a \chi^{\a - 1}_{t u} \wedge \omega^{n - \a} - \b f \chi^{\b - 1}_{t u} \wedge \omega^{n - \b}) \right) dt \\
	&\geq c_1 \int_M e^{- p u} \sqrt{- 1} \p u \wedge \bpartial u \wedge \omega^{n - 1},
\end{aligned}
\end{equation}
and for $0 \leq t \leq \frac{1}{2}$,
\begin{equation}
\label{gche-quotient-integral-9}
\begin{aligned}
	\a \chi^{\a - 1}_{t u} \wedge \omega^{n - \a} - \b f \chi^{\b - 1}_{t u} \wedge \omega^{n - \b} > c_2 \chi^{\a - 1}_{t u} \wedge \omega^{n - \a}.
\end{aligned}
\end{equation}
Therefore, combining \eqref{gche-quotient-integral-6} - \eqref{gche-quotient-integral-9} and Lemma~\ref{gche-lemma-iteration}, choosing small $\epsilon_1$ and then sufficiently large $p$, we can obtain
\begin{equation}
\label{gche-quotient-integral-10}
\begin{aligned}
	I \geq \frac{c_1 p}{2} \int_M e^{- p u} \sqrt{- 1} \p u \wedge \bpartial u \wedge \omega^{n - 1} - \frac{C_8}{p} \int_M e^{- p u} \omega^n.
\end{aligned}
\end{equation}

\end{proof}

\bigskip

\section{Hessian equations}
\label{gche-hessian}
\setcounter{equation}{0}
\medskip

In this section, we follow the work of Hou, Ma and Wu~\cite{HouMaWu10} to obtain the second order estimate quadratically dependent on the gradient. Then the blow-up argument of Dinew and Kolodziej~\cite{DK} suffices to prove the gradient estimate. Higher order estimates follow from Evans-Krylov theory and Schauder estimate, which is done by Tosatti, Wang, Weinkove and Yang~\cite{TWWvY2014}. Therefore, we can apply the conitnuity method shown in \cite{TWv10a} and \cite{Sun2013e} to prove the existence of solution.

\begin{lemma}
\label{lemma-hessian-C2}
Let $u \in C^4(M)$ be an admissible solution to Hessian equation~\eqref{hessian-equation} on closed K\"ahler manifold $(M,\omega)$. Then there is a second derivative estimate
\begin{equation}
	\sup_M |\p\bpartial u| \leq C \Big(\sup_M |\nabla u|^2 + 1\Big),
\end{equation}
where $C$ depends on $||F^{\frac{1}{\a}}||_{C^2(M)}$, $\a$ and geometric data.
\end{lemma}


First of all, we need to clarify the derivatives. Assume that $U$ is a $C^2$ Hermitian structure. Let
\begin{equation}
	W(x) = \sup_{\xi\in T^{(1,0)}_x M, |\xi|_g = 1} U(\xi,\bar\xi) .
\end{equation}
$W$ is Lipschitz coninuous but might not be differentiable. In the calculation, we need to understand the formal differentiation of $W$, which should independent on the choice of local coordinate charts. In other words, the formal derivatives of $W$ are covariant derivatives.

Suppose that $W(x)$ is achieved at some unit $\xi\in T^{(1,0)}_x M$. For any $\zeta \in T^{(1,0)}_x M$ and $t \in \bfR$,
\begin{equation}
	\frac{U(\xi + t\zeta, \bar\xi + t\bar\zeta)}{g(\xi + t \zeta, \bar\xi + t\bar\zeta)} \leq U(\xi, \bar\xi) .
\end{equation}
As an immediate result,
\begin{equation}
	\frac{d}{d t} \left( \frac{U(\xi + t\zeta, \bar\xi + t\bar\zeta)}{g(\xi + t\zeta, \bar\xi + t\bar\zeta)}\right)\Bigg|_{t = 0} = 0 ,
\end{equation}
that is,
\begin{equation}
\label{hessian-formual-variation-1st}
	U(\xi,\bar\zeta) + U(\zeta,\bar\xi) = U(\xi,\bar\xi) [g(\xi,\bar\zeta) + g(\zeta,\bar\xi)] .
\end{equation}

Considering the eigenvalues of $U$ with respect to $g$, we can sort them in order $\lambda_1 \geq \lambda_2 \geq \cdots \geq \lambda_n$. Suppose that $W(p)$ is attained at some unit $\eta \in T^{(1,0)}_p M$. There are two cases $\lambda_1 > \lambda_2$ and $\lambda_1 = \lambda_2$.

{\bf Case 1.} If $\lambda_1 > \lambda_2$ at $p$, then the inequality holds true in a small neighborhood $B_\delta$ of $p$. Thus, we can find a smooth $(1,0)$-vector field $\eta$ in $B_\delta$ such that for any $x \in B_\delta$,
\begin{equation}
	g(\eta,\bar\eta ) = 1, \text{ and }\qquad U(\eta,\bar\eta) = \sup_{\xi \in T^{(1,0)}_x M , |\xi|_g = 1} U(\xi , \bar\xi) .
\end{equation}
For any real vector field $\gamma$ on $B_\delta$,
\begin{equation}
	U(\eta , \nabla_\gamma \bar\eta) + U(\nabla_\gamma \eta , \bar\eta)
	= U (\eta,\bar\eta) [g(\eta,\nabla_\gamma \bar\eta) + g(\nabla_\gamma \eta, \bar\eta)] 
	= 0 .
\end{equation}
Thus
\begin{equation}
\label{formal-derivative-1-1}
	\gamma U(\eta, \bar\eta) = \nabla U (\eta, \bar\eta, \gamma) + U(\nabla_\gamma \eta, \bar\eta) + U(\eta, \nabla_\gamma \bar\eta) = \nabla U (\eta, \bar\eta, \gamma),
\end{equation}
which is real valued. By decomposition,
\begin{equation}
\label{formal-derivative-1-2}
	\p_i U(\eta, \bar\eta) = \nabla U \left(\eta, \bar\eta, \frac{\p}{\p z^i}\right) .
\end{equation}

Now we calculate the second derivatives.
\begin{equation}
\label{formal-derivative-2-1}
\begin{aligned}
	\bpartial_i\p_i U(\eta,\bar\eta) =&\, \bpartial_i \left[\nabla U\left(\eta ,\bar\eta, \frac{\p}{\p z^i}\right)\right] \\
	=&\, \nabla^2 U \left(\eta, \bar\eta, \frac{\p}{\p z^i}, \frac{\p}{\p\bar z^i}\right) + \p_i \left[ U \left(\nabla_{\frac{\p}{\p\bar z^i}} \eta, \bar\eta \right) + U \left(\eta, \nabla_{\frac{\p}{\p\bar z^i}}\bar\eta \right)\right] \\
	 &\, - U \left(\nabla_{\frac{\p}{\p z^i}} \nabla_{\frac{\p}{\p\bar z^i}} \eta, \bar\eta\right) - U \left(\nabla_{\frac{\p}{\p\bar z^i}} \eta, \nabla_{\frac{\p}{\p z^i}}\bar\eta\right) \\
	 &\, - U \left(\nabla_{\frac{\p}{\p z^i}}\eta, \nabla_{\frac{\p}{\p\bar z^i}}\bar\eta \right) - U \left(\eta, \nabla_{\frac{\p}{\p z^i}}\nabla_{\frac{\p}{\p\bar z^i}} \bar\eta\right) \\
	=&\, \nabla^2 U \left(\eta, \bar\eta, \frac{\p}{\p z^i}, \frac{\p}{\p\bar z^i}\right) - U \left(\nabla_{\frac{\p}{\p z^i}}\eta, \nabla_{\frac{\p}{\p\bar z^i}}\bar\eta \right) - U \left(\nabla_{\frac{\p}{\p\bar z^i}} \eta, \nabla_{\frac{\p}{\p z^i}}\bar\eta\right)  \\
	 &\, + U (\eta , \bar\eta)  g \left( \nabla_{\frac{\p}{\p\bar z^i}} \eta, \nabla_{\frac{\p}{\p z^i}} \bar\eta\right) + U (\eta, \bar\eta)  g \left(\nabla_{\frac{\p}{\p z^i}} \eta, \nabla_{\frac{\p}{\p\bar z^i}} \bar\eta\right) + Y, 
\end{aligned}
\end{equation}
where
\begin{equation}
\begin{aligned}
	Y =&\,  - \frac{1}{2} U \left(\nabla_{\frac{\p}{\p z^i}} \nabla_{\frac{\p}{\p\bar z^i}} \eta, \bar\eta\right) + \frac{1}{2} U \left(\eta , \nabla_{\frac{\p}{\p\bar z^i}} \nabla_{\frac{\p}{\p z^i}} \bar \eta\right) \\
	 &\,  - \frac{1}{2} U \left(\eta, \nabla_{\frac{\p}{\p z^i}}\nabla_{\frac{\p}{\p\bar z^i}} \bar\eta\right) + \frac{1}{2} U \left( \nabla_{\frac{\p}{\p\bar z^i}} \nabla_{\frac{\p}{\p z^i} } \eta, \bar \eta\right) .
\end{aligned}
\end{equation}
By the definition of $\eta$, 
\begin{equation}
\label{formal-derivative-2-2}
	\bpartial_i\p_i U(\eta,\bar\eta) \geq \nabla^2 U \left(\eta, \bar\eta, \frac{\p}{\p z^i}, \frac{\p}{\p\bar z^i}\right) + Y .
\end{equation}
But by the symmetry of $U$, $\bar Y = - Y$. 
Thus we have
\begin{equation}
\label{formal-derivative-2-3}
	\bpartial_i\p_i U(\eta,\bar\eta) 
	\geq \mathfrak{Re}\left\{ \nabla^2 U \left(\eta, \bar\eta, \frac{\p}{\p z^i}, \frac{\p}{\p\bar z^i}\right) \right\} .
\end{equation}

{\bf Case 2.} If $\lambda_1 = \lambda_2$ at $p$ we can lift up $\lambda_1$ locally by perturbation method and hence transform this case into the first case. Then a standard limiting argument can be applied if necessary.

Now we begin to prove the second order estimates.
\begin{proof}[Proof of Lemma~\ref{lemma-hessian-C2}]
We can rewrite the equation \eqref{hessian-equation} as below,
\begin{equation}
\label{hessian-equation-expression}
	S_\a = F : = C^\a_n f\,.
\end{equation}
Expressing $X = \chi_u$ and differentiating \eqref{hessian-equation-expression} twice under a normal coordinate chart, we obtain
\begin{equation}
\label{hessian-C2-formula-S-1}
	\p_l F = \sum_i S_{\a - 1;i} X_{i\bar il} \,,
\end{equation}
and
\begin{equation}
\label{hessian-C2-formula-S-2}
	\bpartial_l\p_l F = - \sum_{i\neq j} S_{\a - 2;ij} X_{i\bar jl} X_{j\bar i\bar l} + \sum_{i\neq j} S_{\a - 2;ij} X_{j\bar jl} X_{i\bar i\bar l} + \sum_i S_{\a - 1;i} X_{i\bar il\bar l} \,.
\end{equation}

Define
\begin{equation}
\label{hessian-test-function}
	H (x , \xi) = \ln \Bigg(\sum_{i,j} X_{i\bar j} \xi^i \bar\xi^j \Bigg) + \varphi (|\nabla u|^2) + \psi (u)
\end{equation}
where
\begin{equation}
\begin{aligned}
	\varphi (s) &= - \frac{1}{2} \ln \left(1 - \frac{s}{2K}\right), &&\quad \text{ for } 0\leq s \leq K -1, \\
	\psi (t) &= - A \ln \left(1 + \frac{t}{2L}\right), &&\quad \text{ for } - L + 1 \leq t \leq 0,
\end{aligned}
\end{equation}
for 
\begin{equation}
	K := \sup_M |\nabla u|^2 + 1,\; L := \sup_M |u| + 1,\; A := 2L(C_0 + 1)
\end{equation}
and $C_0$ is to be determined later. Note that we have
\begin{equation}
	\frac{1}{2K} \geq \varphi' \geq \frac{1}{4K} > 0,\,\,\, \varphi'' = 2(\varphi')^2 > 0 
\end{equation}
and
\begin{equation}
	\frac{A}{L} \geq - \psi' \geq \frac{A}{2L} = C_0 + 1,\,\, \psi'' \geq \frac{2\epsilon}{1 - \epsilon} (\psi')^2,\,\,\, \mbox{for all } \epsilon \leq \frac{1}{2A + 1} .
\end{equation}
Assume that $H$ achieves a maximum at some point $p$ in the unit direction of $\eta$. Around $p$, we choose a normal coordinate chart such that $X_{1\bar 1} \geq X_{2\bar 2} \geq \cdots \geq X_{n\bar n}$.
Without loss of generality, we can also assume that $X_{1\bar 1} > N$ for some fixed constant $N >> 1$. 
Note that
\begin{equation}
	X_{i\bar i} > - (n - 2) X_{1\bar 1}
\end{equation}
for any $1 \leq i\leq n$.

By \eqref{formal-derivative-1-2}, \eqref{formal-derivative-2-2} and a possible limiting argument at $p$,
\begin{equation}
\label{hessian-derivative-1}
	H_i := \frac{X_{1\bar 1 i}}{X_{1\bar 1}} + \varphi' \p_i (|\nabla u|^2) + \psi' u_i = 0
\end{equation}
and
\begin{equation}
\label{hessian-derivative-2}
\begin{aligned}
 H_{i\bar i} :=&\, \frac{\mathfrak{Re}\{X_{1\bar 1i\bar i}\}}{X_{1\bar 1}} - \frac{|X_{1\bar 1i}|^2 }{X^2_{1\bar 1}} + \varphi'' |\p_i(|\nabla u|^2) |^2 \\
	 &\, + \varphi' \bpartial_i\p_i (|\nabla u|^2) + \psi'' |u_i|^2 + \psi' u_{i\bar i}\, \leq 0.
\end{aligned}
\end{equation}
From \eqref{hessian-derivative-1},
\begin{equation}
	\frac{X_{1\bar 1i}}{X_{1\bar 1}} = - \varphi' \p_i (|\nabla u|^2) - \psi' u_i .
\end{equation}
From \eqref{hessian-derivative-2},
\begin{equation}
\label{hessian-formula-2-1}
\begin{aligned}
	0 \geq&\, \frac{1}{X_{1\bar 1}}\mathfrak{Re}\left\{X_{i\bar i 1\bar 1} + R_{i\bar i1\bar 1} X_{1\bar 1} - R_{1\bar 1i\bar i} X_{i\bar i} - G_{1\bar 1i\bar i}\right\}  - \frac{|X_{1\bar 1i}|^2}{X^2_{1\bar 1}} \\
	 &\, + \varphi'' |\p_i(|\nabla u|^2) |^2 + \varphi' \bpartial_i\p_i (|\nabla u|^2) + \psi'' |u_i|^2 + \psi' u_{i\bar i} ,
\end{aligned}
\end{equation}
where
\begin{equation}
    	G_{i\bar ij\bar j} = \chi_{j\bar ji\bar i} - \chi_{i\bar ij\bar j} + \sum_p R_{j\bar ji\bar p}\chi_{p\bar i} -\sum_p R_{i\bar ij\bar p}\chi_{p\bar j}.
\end{equation}
Combining \eqref{hessian-formula-2-1}, \eqref{hessian-C2-formula-S-1} and \eqref{hessian-C2-formula-S-2} and the concavity of $S^{1/\a}_\a$, we have
\begin{equation}
\label{hessian-formula-2-2}
\begin{aligned}
 	0 \geq&\, \frac{1}{X_{1\bar 1}} \left( \bpartial_1\p_1 F - \frac{\a - 1}{\a} \p_1F \bpartial_1F\right) +  \frac{1}{X_{1\bar 1}} \sum_{i\neq j} S_{\a - 2;ij} X_{i\bar j1} X_{j\bar i\bar 1}  \\
	 &\,+ \frac{1}{X_{1\bar 1}}\sum_i S_{\a - 1;i}\mathfrak{Re}\left\{ R_{i\bar i1\bar 1} X_{1\bar 1} - R_{1\bar 1i\bar i} X_{i\bar i} - G_{1\bar 1i\bar i}\right\} \\
	 &\, - \frac{1}{X^2_{1\bar 1}} \sum_i S_{\a - 1;i} |X_{1\bar 1i}|^2  + \varphi''\sum_i S_{\a - 1;i}  |\p_i(|\nabla u|^2)|^2 \\
	 &\,   + \varphi' \sum_i S_{\a - 1;i}\bpartial_i\p_i (|\nabla u|^2) + \psi''\sum_i S_{\a - 1;i} |u_i|^2 + \psi' \sum_i S_{\a - 1;i} u_{i\bar i} .
\end{aligned}
\end{equation}
Direct calculation shows that
\begin{equation}
\label{hessian-formula-2-3}
	\p_i (|\nabla u|^2) = \sum_j (u_j u_{\bar ji} + u_{ji} u_{\bar j}) 
\end{equation}
and
\begin{equation}
\label{hessian-formula-2-4}
\begin{aligned}
	\bpartial_i\p_i (|\nabla u|^2) &= \sum_j (u_{j\bar i} u_{\bar ji} + u_{ji} u_{\bar j\bar i} + u_{ji\bar i} u_{\bar j} + u_j u_{\bar ji\bar i}) .
\end{aligned}
\end{equation}
Aslo note that
\begin{equation}
\label{hessian-formula-2-5}
	u_{ji\bar i} = \sum_k R_{i\bar i j\bar k} u_k + u_{i\bar i j},
\end{equation}
and
\begin{equation}
\label{hessian-formula-2-6}
	u_{\bar j i\bar i} = u_{i\bar i\bar j}.
\end{equation}
Combining \eqref{hessian-formula-2-2}, \eqref{hessian-formula-2-4}, \eqref{hessian-formula-2-5} and \eqref{hessian-formula-2-6}, and throwing away some nonnegative terms,
\begin{equation}
\label{hessian-formula-2-7}
\begin{aligned}
  	0 
	 \geq&\, \frac{1}{X_{1\bar 1}} \left( \bpartial_1\p_1 F - \frac{\a - 1}{\a} \p_1F \bpartial_1 F\right) +  \frac{1}{X_{1\bar 1}} \sum_{i\neq j} S_{\a - 2;ij} X_{i\bar j1} X_{j\bar i\bar 1}  \\
	 &\,+ \frac{1}{X_{1\bar 1}}\sum_i S_{\a - 1;i}R_{i\bar i1\bar 1}\left( X_{1\bar 1} - X_{i\bar i}\right) - \frac{1}{X_{1\bar 1}}\sum_i S_{\a - 1;i}\mathfrak{Re}\left\{ G_{1\bar 1i\bar i}\right\} \\
	 &\, - \frac{1}{X^2_{1\bar 1}} \sum_i S_{\a - 1;i} |X_{1\bar 1i}|^2  + \varphi''\sum_i S_{\a - 1;i}  |\p_i(|\nabla u|^2)|^2  + \varphi' \sum_i S_{\a - 1;i} X^2_{i\bar i} \\
	 &\,   - 2  \varphi' \sum_i S_{\a - 1;i}  X_{i\bar i} \chi_{i\bar i} + \varphi' \sum_{i,j,k} S_{\a - 1;i} R_{i\bar i j\bar k} u_k  u_{\bar j} + 2 \varphi' \sum_{j} \mathfrak{Re}\{\p_j F u_{\bar j}\} \\
	 &\, - 2\varphi' \sum_{i,j} S_{\a - 1;i}  \mathfrak{Re}\{\chi_{i\bar i j} u_{\bar j}\}  + \psi''\sum_i S_{\a - 1;i} |u_i|^2 + \psi' \sum_i S_{\a - 1;i} u_{i\bar i}  .
\end{aligned}
\end{equation}
We have
\begin{equation}
\label{hessian-formula-2-8}
\begin{aligned}
	 &\, \frac{1}{X_{1\bar 1}}\sum_i S_{\a - 1;i}R_{i\bar i1\bar 1}\left( X_{1\bar 1} - X_{i\bar i}\right) - \frac{1}{X_{1\bar 1}}\sum_i S_{\a - 1;i}\mathfrak{Re}\left\{ G_{1\bar 1i\bar i}\right\} \\
	 \geq&\, \inf_k (R_{1\bar 1k\bar k}) \frac{1}{X_{1\bar 1}} \sum_i S_{\a - 1;i} (X_{1\bar 1} - X_{i\bar i}) - \sup_k (\mathfrak{Re} \{G_{1\bar 1k\bar k}\})\frac{1}{X_{1\bar 1}} \sum_i S_{\a - 1;i} \\
	 \geq&\, \inf_k (R_{1\bar 1k\bar k}) \sum_i S_{\a - 1;i} - \inf_k (R_{1\bar 1k\bar k}) \a F - \sup_k (\mathfrak{Re} \{G_{1\bar 1k\bar k}\})\frac{1}{X_{1\bar 1}} \sum_i S_{\a - 1;i} ,
\end{aligned}
\end{equation}
and
\begin{equation}
\label{hessian-formula-2-9}
\begin{aligned}
	- 2  \varphi' \sum_i S_{\a - 1;i}  X_{i\bar i} \chi_{i\bar i} 
	\geq&\, - \frac{1}{2} \varphi' \sum_i S_{\a - 1;i} X^2_{i\bar i} - \frac{\sup_k \chi^2_{k\bar k}}{K} \sum_i S_{\a - 1;i} .
\end{aligned}
\end{equation}
Also,
\begin{equation}
\label{hessian-formula-2-10}
\begin{aligned}
	&\, \varphi' \sum_{i,j,k} S_{\a - 1;i} R_{i\bar i j\bar k} u_k  u_{\bar j} +  2\varphi' \sum_{j} \mathfrak{Re} \{\p_j F u_{\bar j} \} - 2 \varphi' \sum_{i,j} S_{\a - 1;i}  \mathfrak{Re}\{\chi_{i\bar i j} u_{\bar j} \} \\
	\geq&\, - \frac{1}{2} \sup_{j,k,l} | R_{l\bar lj\bar k}| \sum_i S_{\a - 1;i} - \frac{|\nabla F||\nabla u|}{K} - \frac{\sup_{k}(\sum_j|\chi_{k\bar kj}|)|\nabla u|}{K} \sum_i S_{\a - 1;i} .
\end{aligned}
\end{equation}
There are constants $\Lambda \geq \lambda > 0$ such that $\chi - \lambda \omega, \Lambda \omega - \chi \in \Gamma^\a_\omega$, so by G\r{a}rding's inequality,
\begin{equation}
\label{hessian-formula-2-garding}
\begin{aligned}
	\psi' \sum_i S_{\a - 1;i} u_{i\bar i} &= \psi' \left[\sum_i S_{\a - 1;i} X_{i\bar i} - \sum_i S_{\a - 1;i} (\chi_{i\bar i} - \lambda) - \lambda \sum_i S_{\a - 1;i}\right] \\
	&\geq - 2(C_0 + 1) \a F + (C_0 + 1) \lambda \sum_i S_{\a - 1;i} .
\end{aligned}
\end{equation}
Substituting \eqref{hessian-formula-2-8}, \eqref{hessian-formula-2-9}, \eqref{hessian-formula-2-10} and \eqref{hessian-formula-2-garding} into \eqref{hessian-formula-2-7},
\begin{equation}
\label{hessian-formula-2-11}
\begin{aligned}
	 0 
	 \geq&\, - C_1 - C_2 \sum_i S_{\a - 1;i} - 2(C_0 + 1) \a F - \frac{1}{X^2_{1\bar 1}} \sum_i S_{\a - 1;i} |X_{1\bar 1i}|^2  \\
	 &\,  +  \frac{1}{X_{1\bar 1}} \sum_{i\neq j} S_{\a - 2;ij} X_{i\bar j1} X_{j\bar i\bar 1} + \varphi''\sum_i S_{\a - 1;i}  |\p_i(|\nabla u|^2)|^2 \\
	 &\,   + \frac{\varphi'}{2} \sum_i S_{\a - 1;i} X^2_{i\bar i} + \psi''\sum_i S_{\a - 1;i} |u_i|^2 + (C_0 + 1) \lambda \sum_i S_{\a - 1;i} .
\end{aligned}
\end{equation}
Choosing
\begin{equation}
C_0 = \frac{1}{\lambda} \left(\frac{C_1}{\a C^\a_n \min_M f^{\frac{\a - 1}{\a}}}  + C_2\right) ,
\end{equation}
we obtain
\begin{equation}
\label{hessian-formula-2-12}
\begin{aligned}
	 0 \geq&\, - 2(C_0 + 1) \a F - \frac{1}{X^2_{1\bar 1}} \sum_i S_{\a - 1;i} |X_{1\bar 1i}|^2  \\
	 &\,  +  \frac{1}{X_{1\bar 1}} \sum_{i\neq j} S_{\a - 2;ij} X_{i\bar j1} X_{j\bar i\bar 1} + \varphi''\sum_i S_{\a - 1;i}  |\p_i(|\nabla u|^2)|^2 \\
	 &\,   + \varphi' \sum_i S_{\a - 1;i} X^2_{i\bar i} + \psi''\sum_i S_{\a - 1;i} |u_i|^2 + \lambda \sum_i S_{\a - 1;i} .
\end{aligned}
\end{equation}

Define
\begin{equation}
\label{hessian-defintion-delta} 
	\delta = \frac{1}{1 + 2A} = \frac{1}{1 + 4L(C_0 + 1)},
\end{equation}
and then we have two cases in consideration.

{\bf Case 1.} $\lambda_n < - \delta \lambda_1$

The first derivative \eqref{hessian-derivative-1} tells us that
\begin{equation}
\label{hessian-formula-case-1-1}
\begin{aligned}
	 0 \geq&\, - 2(C_0 + 1) \a F -  \sum_i S_{\a - 1;i} |\varphi'\p_i (|\nabla u|^2)  + \psi' u_i|^2  \\
	 &\,  +  \frac{1}{X_{1\bar 1}} \sum_{i\neq j} S_{\a - 2;ij} X_{i\bar j1} X_{j\bar i\bar 1}  + \varphi''\sum_i S_{\a - 1;i}  |\p_i(|\nabla u|^2)|^2 \\
	 &\,   + \frac{\varphi'}{2} \sum_i S_{\a - 1;i} X^2_{i\bar i} + \psi''\sum_i S_{\a - 1;i} |u_i|^2 + \lambda \sum_i S_{\a - 1;i} \\
	 \geq&\,  - 2(C_0 + 1) \a F - 8 (C_0 + 1)^2 K \sum_i S_{\a - 1;i}  + \frac{1}{8K} \sum_i S_{\a - 1;i} X^2_{i\bar i} .
\end{aligned}
\end{equation}
Therefore,
\begin{equation}
\label{hessian-formula-case-1-2}
\begin{aligned}
	 0 \geq&\,  - 2(C_0 + 1) \a F - 8 (C_0 + 1)^2 K \sum_i S_{\a - 1;i}  + \frac{1}{8K} S_{\a - 1;n} X^2_{n\bar n} \\
	 \geq&\,  - 2(C_0 + 1) \a F - 8 (C_0 + 1)^2 K \sum_i S_{\a - 1;i}  + \frac{1}{8nK} X^2_{n\bar n} \sum_i S_{\a - 1;i} .
\end{aligned}
\end{equation}
It follows that
\begin{equation}
\label{hessian-formula-case-1-3}
	X^2_{1\bar 1} \leq \frac{8nK}{\delta^2} \left[ \frac{2 (C_0 + 1) \a \sup_M F}{(n - \a + 1) C^{\a - 1}_n \inf_M f^{\frac{\a - 1}{\a}}} + 8(C_0 + 1)^2 K\right]
\end{equation}
and thus
\begin{equation}
	X_{1\bar 1} \leq \frac{4K}{\delta} \sqrt{\frac{ n (C_0 + 1) \sup_M f}{ \inf_M f^{\frac{\a - 1}{\a}}} + 4 n (C_0 + 1)^2 } .
\end{equation}

{\bf Case 2.} $\lambda_n \geq - \delta \lambda_1$.

Let
\begin{equation}
\label{hessian-definition-I}
	I = \{i\in\{1,\cdots,n\}\,|\, S_{\a - 1;i} > \delta^{-1} S_{\a - 1;1}\}.
\end{equation}
Obviously, $1\not\in I$.

From \eqref{hessian-formula-2-12}, 
\begin{equation}
\label{hessian-formula-case-2-1}
\begin{aligned}
	0 
	 \geq&\, - 2(C_0 + 1) \a F - \frac{1}{X^2_{1\bar 1}} \sum_i S_{\a - 1;i} |X_{1\bar 1i}|^2 + \varphi''\sum_i S_{\a - 1;i}  |\p_i(|\nabla u|^2)|^2  \\
	 &\, + \frac{1 - \delta}{1 + \delta} \frac{1}{X^2_{1\bar 1}} \sum_{i\in I} S_{\a - 1;i} \left(|X_{1\bar 1i} |^2 + 2\mathfrak{Re}\{X_{1\bar 1i}\bar b_i\} \right) + \frac{\varphi'}{2} \sum_i S_{\a - 1;i} X^2_{i\bar i} \\
	 &\, + \psi''\sum_i S_{\a - 1;i} |u_i|^2 + \lambda \sum_i S_{\a - 1;i}
\end{aligned}
\end{equation}
where $b_i =\chi_{i\bar 1 1} - \chi_{1\bar 1i}$.

Let us treat the indeces which are in $I$. The first derivative \eqref{hessian-derivative-1} tells us that
\begin{equation}
\label{hessian-formula-case-2-2}
\begin{aligned}
	&\, - \frac{1}{X^2_{1\bar 1} }\sum_{i\in I} S_{\a - 1;i} |X_{1\bar 1i}|^2 + \varphi'' \sum_{i\in I} S_{\a - 1;i} |\p_i(|\nabla u|^2)|^2\\
	&\, + \frac{1 - \delta}{1 + \delta} \frac{1}{X^2_{1\bar 1}} \sum_{i\in I} S_{\a - 1;i} \left(|X_{1\bar 1i} |^2 + 2\mathfrak{Re}\{X_{1\bar 1i}\bar b_i\} \right) + \psi'' \sum_{i\in I} S_{\a - 1;i} |u_i|^2 \\
	\geq&\,  - \frac{1}{X^2_{1\bar 1} }\sum_{i\in I} S_{\a - 1;i} |X_{1\bar 1i}|^2 + 2 \sum_{i\in I} S_{\a - 1;i} \left(\delta\left|\frac{X_{1\bar 1 i}}{X_{1\bar 1}}\right|^2 - \frac{\delta}{1 - \delta} |\psi' u_i|^2\right)\\
	&\, + \frac{1 - \delta}{1 + \delta} \frac{1}{X^2_{1\bar 1}} \sum_{i\in I} S_{\a - 1;i} \left(|X_{1\bar 1i} |^2 + 2\mathfrak{Re}\{X_{1\bar 1i}\bar b_i\} \right) + \psi'' \sum_{i\in I} S_{\a - 1;i} |u_i|^2 \\
	\geq	&\, \frac{\delta^2}{X^2_{1\bar 1}} \sum_{i \in I} S_{\a - 1;i} |X_{1\bar 1i}^2| - \frac{1 - \delta}{(1 + \delta) \delta^2} \frac{1}{X^2_{1\bar 1}} \sum_{i \in I} S_{\a - 1;i} |b_i|^2 .
\end{aligned}
\end{equation}
Now let us treat the indeces which are not in I. Following \eqref{hessian-formula-case-1-1}
\begin{equation}
\label{hessian-formula-case-2-3}
\begin{aligned}
	 &\, - \frac{1}{X^2_{1\bar 1}} \sum_{i\not\in I} S_{\a - 1;i} |X_{1\bar 1i}|^2 + \varphi''\sum_{i\not\in I} S_{\a - 1;i}  |\p_i(|\nabla u|^2)|^2  \\	 
	 &\qquad\geq - 8 (C_0 + 1)^2 K \max_{i\not\in I} S_{\a - 1;i}  \geq -\frac{8 (C_0 + 1)^2 K}{\delta} S_{\a - 1;1} .
\end{aligned}
\end{equation}
Substituting the inequalities \eqref{hessian-formula-case-2-2} and \eqref{hessian-formula-case-2-3} into \eqref{hessian-formula-case-2-1},
\begin{equation}
\label{hessian-formula-case-2-4}
\begin{aligned}
	 0 \geq&\, - 2(C_0 + 1) \a F + \frac{\delta^2}{X^2_{1\bar 1}} \sum_{i \in I} S_{\a - 1;i} |X_{1\bar 1i}^2| - \frac{1 - \delta}{(1 + \delta) \delta^2 X^2_{1\bar 1}} \sum_{i \in I} S_{\a - 1;i} |b_i|^2 \\
	 &\,-\frac{8 (C_0 + 1)^2 K}{\delta} S_{\a - 1;1}  + \frac{1}{8K} \sum_i S_{\a - 1;i} X^2_{i\bar i} + \lambda \sum_i S_{\a - 1;i} \\
	 \geq&\,  - 2(C_0 + 1) \a F - \frac{1 - \delta}{(1 + \delta) \delta^2 X^2_{1\bar 1}} \sum_k |b_k|^2 \sum_{i \in I} S_{\a - 1;i}  -\frac{8 (C_0 + 1)^2 K}{\delta} S_{\a - 1;1} \\
	 &\,  + \frac{1}{8K} \sum_i S_{\a - 1;i} X^2_{i\bar i} + \lambda \sum_i S_{\a - 1;i} .
\end{aligned}
\end{equation}
Without loss of generality, we can assume that
\begin{equation}
\label{hessian-formula-case-2-5}
	X^2_{1\bar 1} \geq \max\left\{\frac{128 (C_0 + 1)^2 K^2}{\delta} , \frac{2(1 - \delta)}{\lambda (1 + \delta) \delta^2 } \sum_k |b_k|^2\right\},
\end{equation}
and hence
\begin{equation}
\label{hessian-formula-case-2-6}
	2 (C_0 + 1) \a F \geq \frac{1}{16K} \sum_i S_{\a - 1;i} X^2_{i\bar i} + \frac{\lambda}{2} \sum_i S_{\a - 1;i} .
\end{equation}
Applying Lemma 2.2 in \cite{HouMaWu10}, we obtain a lower bound for $S_{\a - 1;1}$ and then the desired bound for $X_{1\bar 1}$.

\end{proof}


\bigskip

















\begin{thebibliography}{99}


\bibitem{Blocki2005a}
Z. Blocki,
{\em On uniform estimate on Calabi-Yau theorem},
Sci. China Ser. A, {\bf 48} (2005), 244--247.

\bibitem{Chen00b}
X.-X. Chen,
{\em On the lower bound of the Mabuchi energy and its application},
Int. Math. Res. Notices \textbf{12}  (2000),  607--623


\bibitem{DK}
S. Dinew and S. Kolodziej,
{\em Liouville and Calabi-Yau type theorems for complex Hessian equations},
arXiv: 1203.3995.




\bibitem{Donaldson99a}
S. K. Donaldson,
{\em Moment maps and diffeomorphisms},
Asian J. Math. {\bf 3} (1999), 1--16.



\bibitem{FLM11}
H. Fang, M.-J. Lai and X.-N. Ma,
{\em On a class of fully nonlinear flows in K\"ahler geometry}
J. Reine Angew. Math. {\bf 653} (2011), 189--220.

\bibitem{Ga77}
P. Gauduchon,
{\em Le th\'eor\`eme de l'excentricit\'e nulle},
C. R. Acad. Sci. Paris {\bf 285} (1977), 387--390.





\bibitem{Garding59}
L. G\r{a}rding, 
{\em An inequality for hyperbolic polynomials}, 
J. Math. Mech.  {\bf 8} (1959), no. 6, 957--965.



\bibitem{GSun12}
B.~ Guan and W.~Sun,
{\em On a class of fully nonlinear elliptic equations on Hermitian manifolds},
Calc.Var. PDE, DOI 10.1007/s00526-014-0810-1



\bibitem{HouMaWu10}
Z.-L. Hou, X.-N. Ma and D.-M. Wu,
{\em A second order estimate for complex Hessian equations on a compact K\"ahler manifold},
Math. Res. Lett. {\bf 17} (2010), no. 3, 547--561.

\bibitem{SW08}
J. Song and B. Weinkove,
{\em On the convergence and singularities of the J-flow with applications to the Mabuchi energy},
Comm. Pure Appl. Math. {\bf 61} (2008), 210--229.



\bibitem{Sun2013e}
W. Sun,
{\em On a class of fully nonlinear elliptic equations on closed Hermitian manifolds},
preprint,  arXiv:1310.0362.


\bibitem{Sun2014e}
W. Sun,
{\em On a class of fully nonlinear elliptic equations on closed Hermitian manifolds II: $L^\infty$ estimate},
preprint, arXiv:1407.7630.


\bibitem{Szekelyhidi2014b}
G. Sz\'ekelyhidi,
{\em Fully non-linear elliptic equations on compact Hermitian manifolds},
preprint, arXiv:1501.02762.

\bibitem{TWWvY2014}
V. Tosatti, Y. Wang, B. Weinkove and X.-K. Yang,
{\em $C^{2,\a}$ estimates for nonlinear elliptic equations in complex and almost complex geometry},
to appear in Calc. Var. PDE.


\bibitem{TWv10a}
V. Tosatti and B. Weinkove,
{\em Estimates for the complex Monge-Amp\`ere equation on Hermitian and
balanced manifolds},
Asian J. Math. {\bf 14} (2010), 19--40.

\bibitem{TWv10b}
V. Tosatti and B. Weinkove,
{\em The complex Monge-Amp\`ere equation on compact Hermitian manifolds},
J. Amer. Math. Soc. {\bf 23} (2010), 1187--1195.




\bibitem{Yau78}
S.-T. Yau,
{\em On the Ricci curvature of a compact K\"ahler manifold and the complex
Monge-Amp\`ere equation. I.}
Comm. Pure Appl. Math. {\bf 31} (1978), 339--411.


\bibitem{Zhang2015}
D.-K. Zhang,
{\em Hessian equations on closed Hermitian manifolds},
preprint, arXiv:1501.03553.

\end{thebibliography}
\end{document}